\documentclass[a4paper,draft,reqno,12pt]{amsart}
\usepackage[english]{babel}
\usepackage{amsmath}
\usepackage{amssymb}
\usepackage{amscd}
\usepackage{amsthm}
\usepackage{euscript}

\usepackage{mathrsfs}

\usepackage{graphicx}
\newtheorem{theorem}{Theorem}[section]
\newtheorem{lemma}[theorem]{Lemma}
\newtheorem{prop}[theorem]{Proposition}

\newtheorem{corollary}{Corollary}
\theoremstyle{definition}
\newtheorem{definition}{Definition}
\newtheorem{ex}{Example}
\theoremstyle{remark}
\newtheorem {re}{Remark}

\DeclareMathOperator{\Aut}{Aut}

\DeclareMathOperator{\td}{tr.deg.}

\def\Ker{{\rm Ker}\,}

\def\LND{{\rm LND}\,}
\def\ML{{\rm ML}\,}

\def\CC{{\mathbb C}}
\def\KK{{\mathbb K}}

\def\ZZ{{\mathbb Z}}

\def\NN{{\mathbb N}}
\def\QQ{{\mathbb Q}}

\def\AA{{\mathbb A}}

\begin{document}

\date{}

\author{Ilya Boldyrev}
\address{Lomonosov Moscow State University, Faculty of Mechanics and Mathematics, Department of Higher Algebra, Leninskie Gory 1, Moscow, 119991 Russia; \newline
and
\newline
HSE University, Faculty
of Computer Science, Pokrovsky Boulevard 11, Moscow, 109028
Russia}\email{boldyrev.i.al@gmail.com}

\title[Makar-Limanov invariants of affine toric varieties]{Makar-Limanov invariants of nonnormal affine toric varieties}

\markboth{Ilya Boldyrev}{Makar-Limanov invariants of nonnormal affine toric varieties}



\thanks{Supported by the grant RSF-DST 22-41-02019.}

\maketitle

\begin{abstract}
We prove the equality of the Makar-Limanov invariant and the modified Makar-Limanov invariant in the case of not necessary normal affine toric varieties. Also we give a combinatorial description of these invariants.
\end{abstract}

\section{Introduction}
Let $\KK$ be an algebraically closed field of characteristic zero. Let us consider an irreducible algebraic variety $X$ over $\KK$. If an algebraic torus $T\cong(\KK^\times)^n$ acts on $X$ with an open orbit, then $X$ is called {\it toric}. Often the definition of a toric variety includes the condition of normality, but in this paper normality is not assumed.

Let $B$ be an affine $\KK$-domain.
The set of locally nilpotent 
$\KK$-derivations of $B$ is denoted by $\LND(B)$.
We denote the kernel of a locally nilpotent derivation $D$ by $\Ker D$. 
The {\it Makar-Limanov invariant} of $B$, denoted 
by $\ML(B)$, is defined as
$$
\ML(B):=\bigcap_{D \in \LND(B)} \Ker~D.
$$ 
The Makar-Limanov invariant has been a powerful 
tool for solving some major problems in affine algebraic geometry
like the Linearization Problem \cite[pp. 195--204]{F}. 

Let us denote by $\KK^{[n]}$ the algebra of polynomials in $n$ variables over a field $\KK$. When $\KK$ is an algebraically closed field of characteristic zero, the Makar-Limanov invariant also gives a characterization  of 
$\KK^{[1]}$, i.e., for an affine 
$\KK$-domain $B$ with $\td_{\KK}B=1$ we have $B=\KK^{[1]}$ if and only if $\ML(B)=\KK$, see \cite[Lemma 2.3]{CML}. 
However, the triviality of the Makar-Limanov invariant alone does not characterize
the affine $2$-space, i.e. $\dim B=2$ and $\ML(B)=\KK$ $\nRightarrow$ $B= \KK^{[2]}$, see \cite[Theorem~2.8]{DG}.

Consider the subset $\LND^{*}(B) \subseteq \LND(B)$ defined by 
$$
\LND^{*}(B)=\{D \in \LND(B)~|~Ds=1~ \text{for some} ~ s \in B\}.
$$
An element $s \in B$ with $Ds = 1$ is called {\it slice} for $D$.

Then we define 
$$
\ML^{*}(B):= \bigcap_{D \in \LND^{*}(B)} \Ker~D.
$$ 
This invariant is introduced by G. Freudenburg  in \cite[p. 237]{F}.
We are going to call it the Makar-Limanov--Freudenburg invariant or ML-F invariant.

Now let $B = \KK[X]$ be the algebra of regular functions of some affine toric variety $X$. The algebra $\KK[X]$ admits
a natural grading by the group of characters $\mathfrak{X}(T)$ of the torus $T$. Let us denote by $\LND_{h}(B)$ the subset of all homogeneous (with respect to this grading) derivations from $\LND(B)$.
So on, we can introduce

$$
\ML_{h}(B):= \bigcap_{D \in \LND_{h}(B)} \Ker~D.
$$.

Here are the main results of this paper.

\begin{theorem}\label{m}
Let $X$ be an affine toric variety. Then
\begin{itemize}
\item[{\rm (a)}]
$\ML_h(\KK[X]) = \ML(\KK[X])$;
\item[{\rm (b)}]
if the algebra $\KK[X]$ admits an LND with slice, then 
$$\ML_{h}(\KK[X]) = \ML(\KK[X]) = \ML^{*}(\KK[X]).$$
\end{itemize}
\end{theorem}

The author is grateful to Ivan Arzhantsev, Nikhilesh Dasgupta and Sergey Gaifullin for useful discussions.

\section{Locally nilpotent derivations}

Algebraic $\mathbb{G}_a$-subgroups of $\Aut(X)$ are in correspondence with locally nilpotent derivations (LND) of the algebra of regular functions~$\KK[X]$. Recall that a linear mapping $\delta\colon\KK[X]\rightarrow\KK[X]$ is called a {\it derivation} of $\KK[X]$, if it satisfies the Leibnits rule  $\delta(fg)=f\delta(g)+\delta(f)g$. A derivation~$\delta$ is called {\it locally nilpotent}, if for each $f\in\KK[X]$ there exists a positive integer number $n$ such, that $\delta^n(f)=0$. More information about LNDs one can find in the book \cite{F}. If we fix an LND~$\delta$ we can consider the following operator on $\KK[X]$, which we call the {\it exponent} of~$\delta$:
$$
\exp(\delta)(f)=f+\delta(f)+\frac{\delta^2(f)}{2!}+\frac{\delta^3(f)}{3!}+\ldots
$$
Since $\delta$ is locally nilpotent, this sum is finite. The operator $\exp(\delta)$ gives an automorphism of $\KK[X]$ and hence an automorphism of the variety $X$. If a function $f$  belongs to the kernel of an LND $\delta$, then the mapping~$f\delta$ is also an LND. It is called a {\it replica} of the derivation~$\delta$. Each LND~$\delta$ corresponds to a $\mathbb{G}_a$-subgroup 
$$\mathcal{H}_\delta=\left\{\exp(t\delta)\mid t\in\KK\right\}\subset\Aut(\KK[X]).$$
Moreover, each algebraic $\mathbb{G}_a$-subgroup corresponds to an LND.

Let us fix a grading of $\KK[X]$  by a commutative group $G$: 
$$\KK[X]=\bigoplus_{g\in G}\KK[X]_g.$$ 
Then a derivation $\delta$ of $\KK[X]$ is called $G$-homogeneous of degree $g_0$, if for every $g\in G$ and a homogeneous element $f\in \KK[X]_g$ we have $\delta(f)\in\KK[X]_{g+g_0}$. If $X$ is an affine toric variety, then the algebra $\KK[X]$ admits a natural grading  of the group of characters of the torus $\mathfrak{X}(T)$. In Section~\ref{lnd} we define $\mathfrak{X}(T)$-homogeneous locally nilpotent derivations, corresponding to so called {\it Demazure roots}. Further this derivations play an important role. If we fix a $\mathbb{Z}$-grading of $\KK[X]$, then every LND $\partial$ can be decomposed onto a sum of homogeneous derivations $\partial=\sum_{i=l}^k\partial_i$. Homogeneous components $\partial_l$ and $\partial_k$ in this sum with the minimal and the maximal degrees  are locally nilpotent, see~\cite[Section~3]{FZ}. If we have $\mathbb{Z}^n$-grading, then each LND can be decomposed onto a sum of homogeneous derivations. The convex hull of degrees of summands is a polyhedron. Derivations corresponding to its vertices are also locally nilpotent. 

\section{Toric varieties}

In this section we give basic facts about toric varieties. More information about toric varieties one can find in books~\cite{CLSch} and~\cite{Ful}. An irreducible algebraic variety is called {\it toric}, if an algebraic torus $T=(\mathbb{K}^\times)^n$ algebraically acts on it with an open orbit. We can assume the action of $T$ on $X$ to be effective. Note that we do not assume toric variety to be normal. Let $X$ be affine. An affine toric variety $X$ corresponds to a finitely generated monoid $P$ of weights of $T$-semiinvariant regular functions. Let us identify the group of characters $\mathfrak{X}(T)$ with a free abelian group $M=\mathbb{Z}^n$. A vector $m\in\mathbb{Z}^n$ with integer coordinates corresponds to the character~$\chi^m$. Since the open orbit on $X$ is isomorphic to $T$, we have an embedding of algebras of regular functions $\mathbb{K}[X]\hookrightarrow\mathbb{K}[T]$. Identifying the algebra $\mathbb{K}[X]$ with its image we obtain the following subalgebra graded by $P$
$$
\mathbb{K}[X]=\bigoplus_{m\in P}\mathbb{K}\chi^m\subset \bigoplus_{m\in M}\mathbb{K}\chi^m=\mathbb{K}[T].
$$
Let us consider the vector space $M_{\mathbb{Q}}=M\otimes_{\mathbb{Z}}\mathbb{Q}$ over the field of rational numbers. The monoid $P$ generates the cone $\sigma^{\vee}=\mathbb{Q}_{\geq0}P\subset M_{\mathbb{Q}}$. Since~$P$ is finitely generated, the cone $\sigma^\vee$ is a finitely generated polyhedral cone. Since the action of~$T$ on $X$ is effective, the cone $\sigma^\vee$ does not belong to any proper subspace of $M_\QQ$. The variety $X$ is normal if and only if the monoid $P$ is {\it saturated}, i.e. $P=M\cap \sigma^{\vee}$. If $P$ is not saturated, then the monoid $P_{sat}=M\cap \sigma^{\vee}$ we call the {\it saturation}  of the monoid $P$. Elements of $P_{sat}\setminus P$ we call {\it holes} of $P$. Let us give some definition according to \cite{TY}.

\begin{definition}
An element $p$ of the monoid $P$ is called {\it saturation point} of $P$, if the moved cone $p+\sigma^{\vee}$ has no holes, i.e. $(p+\sigma^{\vee})\cap M\subset P$. 
 
A face $\tau$ of the cone $\sigma^{\vee}$ is called {\it almost  saturated}, if there is a saturation point of  $P$ in $\tau$. 
Otherwise $\tau$ is called a {\it nowhere saturated} face.
Below, for convenience, face of codimension 1 is called {\it facet}.
\end{definition}

The lattice of one-parameter subgroups of the torus $T$ we denote by~$N$. The lattice $N$ is a dual lattice to $M$. There is a natural pairing
$M\times N\rightarrow \mathbb{Z}$, which we denote $\langle \cdot,\cdot\rangle$. This pairing can be extended to a pairing between vector spaces $N_\QQ=N\otimes_\ZZ\QQ$ and $M_\QQ$. In the space~$N_\QQ$ we define the cone $\sigma$ dual to $\sigma^\vee$, by the rule
$$
\sigma=\{v\in N_\QQ\mid\forall w\in\sigma^\vee : \langle w,v\rangle\geq 0\}.
$$
The finitely generated polyhedral cone $\sigma$ is {\it pointed}, i.e. it does not contain any nontrivial subspaces.

There is a bijection between $k$-dimensional faces of $\sigma$ and $(n-k)$-dimensional faces of~$\sigma^\vee$. A face $\tau\preccurlyeq\sigma$ corresponds to the face  $\widehat{\tau}=\tau^{\bot}\cap\sigma^\vee\preccurlyeq \sigma^\vee$. Also there is a bijection between $(n-k)$-dimensional faces of $\sigma^\vee$ and $k$-dimensional $T$-orbits on $X$. A face $\widehat{\tau}\preccurlyeq \sigma^\vee$ corresponds to the orbit, which is open in the set of zeros of the ideal 
$$I_{\widehat{\tau}}=\bigoplus_{m\in P\setminus\widehat{\tau}}\KK\chi^m.$$
The composition of these bijections gives a bijection between $k$-demensional faces of the cone $\sigma$ and $k$-dimensional $T$-orbits. The orbit corresponding to a face $\tau$, we denote by~$O_\tau$.

\section{The Makar-Limanov invariant}\label{lnd}

In the paper \cite{D} all $M$-homogeneous LNDs  on normal affine toric varieties are described. Consider an extremal ray $\rho$ of $\sigma$. Let us denote by $n_\rho$ the primitive integer vector on $\rho$. The element $e\in M$ is called a {\it Demazure root, corresponding to $\rho$}, if $\langle e, n_\rho \rangle=-1$ and for each other extremal ray $\xi$ of the cone $\sigma$ it is true $\langle e,n_{\xi}\rangle\geq 0$.
The set of all Demazure roots corresponding to an extremal ray~$\rho$ we denote ${\mathscr R}_\rho$. It is easy to see, that for each extremal ray $\rho$ the set
${\mathscr R}_\rho$ is nonempty. Moreover it is infinite.
Each Demazure root $e\in{\mathscr R}_\rho$ corresponds to the LND $\partial_e$ of the algebra $A=\bigoplus_{m\in P_{sat}}\KK\chi^m$ defined on homogeneous elements by the formula
$$
\partial_e(\chi^m)=\langle m,n_\rho\rangle \chi^{m+e}.
$$
The derivation $\partial_e$ is $M$-homogeneous of degree $e$. The kernel of the derivation $\partial_e$ is the subalgebra $\bigoplus_{m\in M\cap \widehat{\rho}}\KK\chi^m$.
If the subalgebra $\KK[X]\subset A$ is $\partial_e$-invariant, then $\partial_e$ induces an LND of the algebra $\KK[X]$, which we denote $\delta_e$. It is easy to see, that the subalgebra $\KK[X]$ is $\partial_e$-invariant if and only if  $(P+e)\cap P_{sat}\subset P$. 

The following statements from \cite{AKZ} and \cite{GB} respectively play a key role in the proofs of the main results of the article. It describes all homogeneous derivations on $K[X]$ and gives a criterium of existing a well defined honogeneous LND $\delta\in{\mathscr R}_\rho$ for a given $\rho$.

\begin{prop}\cite[Proposition 4.4]{AKZ}\label{ak}
\begin{itemize}
\item[{\rm (a)}]
Any homogeneous derivation $\partial \in Der(A)$ has the form $\partial = \lambda \partial_{\rho,e}$ 
for some $\lambda\in\KK$, $\rho\in N$, and $e\in M$ where
$$
\partial_{\rho,e}(\chi^m)=\langle \rho,m \rangle\chi^{m+e} \quad\forall m\in\sigma^\vee\cap M\,.
$$
\item[{\rm (b)}] Let $$\Sigma^\vee=\sigma^\vee\cup\bigcup_{\rho}{\mathscr R}_\rho\,,$$
where $\rho$ runs through all extremal rays of $\sigma$. Then
$\partial_{\rho,e}(A)\subset A$ if and only if $e\in\Sigma^\vee\cap M$ and either $e\in\sigma^\vee$, or  
$e \in \mathscr{R}_{\rho'}$ and $\rho=\rho'$.
\item[{\rm (c)}]
A homogeneous  derivation $\partial\in Der(A)$ is locally nilpotent if and only if $\partial=\lambda\partial_{\rho_i,e}$
 for  a Demazure root $e\in \mathscr{R}_{\rho}$ and  for some $\lambda\in\KK$. 
\end{itemize}
\end{prop}

\begin{lemma}\cite[Lemma 4]{GB}\label{eu}
Let $X$ be an affine toric variety, $\sigma$ be the corresponding cone and $\rho$ be an extremal ray of~$\sigma$. Denote by $O_{\rho}$ the corresponding orbit. Then the following conditions are equivalent:

1) the facet $\widehat{\rho}$ of the cone $\sigma^{\vee}$ is almost saturated;

2) there exist a Demazure root $e\in {\mathscr R}_\rho$ of the cone $\sigma$ such, that the corresponding derivation $\delta_e$ of the algebra $\mathbb{K}[X]$ is well defined;

3) the orbit $O_{\rho}$ consists of smooth points.
\end{lemma}

\begin{proof}[Proof of Theorem \ref{m}(a)]
It is obvious, that $\ML(X) \subset \ML_h(X)$. We are going to prove the inverse. More specifically, we are going to prove that any LND($\KK[X]$) vanishes on $\ML_h(X)$. Firstly, we prove the following lemmas.

\begin{lemma}\label{int}
$\KK$-algebra $\KK[X]_{norm}$, corresponding to $P_{sat}$, is integral over $\KK[X]$.
\end{lemma}
\begin{proof}
First of all, elements $\chi_i \in \KK[X]_{norm}$, corresponding to primitive integer vectors on extremal rays $r_i$ of $\sigma^{\vee}$, are integer over $\KK[X]$. Indeed, extremal rays of $\sigma^{\vee}$ cannot consist only of holes, because opposite contradict to the fact that $P$ is finitely generated. Thus, for every extremal ray $r_i$ exists $n_i \in \NN$, such that $\chi_i^{n_i} \in P$. On the other hand, for every homogeneous element $\chi$ of $\KK[X]_{norm}$ and some nonnegative $m,m_i, \ldots, m_k$ is true that $\chi^m=\chi_1^{m_1}\ldots\chi_k^{m_k}$. This proves what is needed.
\end{proof}

\begin{lemma}\label{van}
If $S$ is a set of facets of $\sigma^{\vee}$ and $B$ is subalgebra of $\KK[X]_{norm}$, corresponding to the intersection of these facets, then every LND $\delta'$, such that its homogeneous LND-components correspond only to facets from $S$, vanishes on $B$.
\end{lemma}
\begin{proof}
Obviously, it is sufficient to prove it for homogeneous elements of~$B$. Consider the decomposition of $\delta'$ on homogeneous components and an arbitrary homogeneous element $\mu\chi^a \in A$ All homogeneous locally nilpotent components of $\delta'$ vanish on it, and all other homogeneous components map $\mu\chi^a$ in $\mu\lambda\langle \rho,a \rangle\chi^{a+e}$ and we can decompose every such image into the product $(\mu\chi^a)(\lambda\langle \rho,a \rangle\chi^e)$, where the second factor belongs to $\KK[X]_{norm}$ (Proposition \ref{ak}(b)). Thus, $\delta'(\chi^a) = 0 + \ldots + 0 + \mu\chi^a (\lambda_1 \langle \rho_1,a \rangle\chi^{e_1} + \ldots + \lambda_s \langle \rho_s,a \rangle\chi^{e_s})$ is divided by $\mu\chi^a$. Hence, by \cite[Principle 5]{F}, $\delta'(\mu\chi^a) = 0$.
\end{proof}
Let $\delta$ be an LND on $\KK[X]$. Proposition \ref{ak}(a) implies that every homogeneous component of $\delta$ has the form $\lambda \langle \rho,m \rangle\chi^{m+e}$ for some $\lambda\in\KK$, $\rho\in N$, and $e\in M$. Then we can extend our derivation $\delta \in Der(\KK[X])$ to the derivation $\delta' \in Der(\KK[X]_{norm})$, define every homogeneous component of $\delta$ on $P_{sat}\backslash P$ by the above formula. Easy to see, that $\delta'$ is the derivation on $\KK[X]_{norm}$. On the other hand, from Lemma 4.2 and Vasconcelos’s Theorem (see \cite{F}, p. 31), it follows, that $\delta'$ is LND on $\KK[X]_{norm}$. Hence, it is sufficient to prove that every such $\delta'$ (i.e. obtained as the extanding of some $\delta \in$ LND($\KK[X]$)) vanishes on the subalgebra, corresponding to the intersection of all almost saturated facets of $\sigma^{\vee}$ and $P_{sat}$. Indeed, if $S$ is a set of $\sigma^{\vee}$ facets, $B$ is subalgebra of $\KK[X]_{norm}$, corresponding to the intersection of these facets, then every LND $\delta'$, such that its homogeneous LND-components corresponding only to facets from $S$, vanishes on $B$.

Now let $S$ to be set of all almost saturated facets of cone $\sigma^{\vee}$. For every $\delta' \in$ LND($\KK[X]_{norm}$), obtained as extanding of some $\delta \in$ LND($\KK[X]$), is true that locally nilpotent homogeneous components of~$\delta'$ corresponds to facets from $S$. It follows from fact that every homogeneous locally nilpotent components of $\delta'$ corresponds to homogeneous locally nilpotent components of $\delta$, for which it is true from Lemma 4.2. Hence, every such $\delta'$ vanishes on $B$ and, therefore, every original derivation $\delta$ from LND($\KK[X]$) vanishes on $B \cap \KK[X] = \ML_h(X)$.
\end{proof} 

\section{The modified Makar-Limanov invariant}

In this section we prove that $\ML(X) = \ML^*(X)$ for an affine toric variety $X$. Also we give combinatorial description of homogeneous LND with slice and prove some corollaries that follow from it.

\begin{lemma}
Let $X$ be an affine toric variety. Then it admits $\LND$ with a slice if and only if it admits homogeneous $\LND$ with homogeneous slice.
\end{lemma}
\begin{proof}
Let $X$ admits LND with slice $\delta$ and $s$ be the corresponding slice. Consider the decomposition of $\delta$ and $s$ into homogeneous components:
$$
\delta(s) = \sum_{i,j} \delta_i(s_j) = 1.
$$

At least one of $\delta_i(s_j) \in \KK \backslash \{0\}$. From Proposition 4.1 (b) it follows that all such $\delta_i$ is LND. Opposite is obvious.
\end{proof}

Now let us study the structure of the cone $\sigma^{\vee}$ if the corresponding variety admits a homogeneous LND with slice. First, introduce several new definitions.

\begin{definition}
A facet $F$ of a cone $\sigma^{\vee}$ is called {\it affine}, if $F$ is saturated and the intersection of all other facets of $\sigma^{\vee}$ has dimension 1.
\end{definition}

\begin{ex}
Let us consider a cone $\sigma^{\vee} = \{ \QQ_{\geq 0}(1,0) + \QQ_{\geq 0}(1,2)\}$ and monoid $P$, which is the intersection of $\sigma^{\vee}$ and $M = \ZZ^2$. Both of one-dimensional facets $F_1$ and $F_2$ of cone $\sigma^{\vee}$ are affine, obviously.
$$
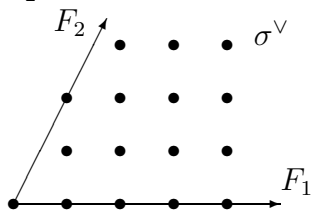
\begin{picture}(120,70)
\put(110,60){$\sigma^{\vee}$}
\put(20,0){\vector(1,0){100}}
\put(20,0){\vector(1,2){35}}
\put(20,0){\circle*{4}}
\put(40,0){\circle*{4}}
\put(40,20){\circle*{4}}
\put(40,40){\circle*{4}}
\put(60,0){\circle*{4}}
\put(60,20){\circle*{4}}
\put(60,40){\circle*{4}}
\put(60,60){\circle*{4}}
\put(80,0){\circle*{4}}
\put(80,20){\circle*{4}}
\put(80,40){\circle*{4}}
\put(80,60){\circle*{4}}
\put(100,0){\circle*{4}}
\put(100,20){\circle*{4}}
\put(100,40){\circle*{4}}
\put(100,60){\circle*{4}}
\put(35,65){$F_2$}
\put(120,5){$F_1$}
\end{picture}
$$
\end{ex}

\begin{ex}
Let us also consider the monoid $P$, consisting of all integer points $(a,b)$ such, that $a \geq 0$, $b \geq 0$, except all points (k, 1). Facet $F_1$ is not affine (because it is not saturated) and $F_2$ is affine.
$$
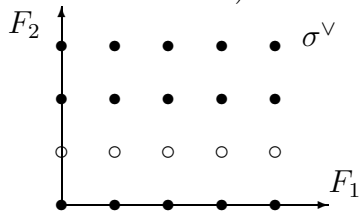
\begin{picture}(120,70)
\put(110,60){$\sigma^{\vee}$}
\put(20,0){\vector(1,0){100}}
\put(20,0){\vector(0,1){75}}
\put(20,0){\circle*{4}}
\put(20,20){\circle{4}}
\put(20,40){\circle*{4}}
\put(20,60){\circle*{4}}
\put(40,0){\circle*{4}}
\put(40,20){\circle{4}}
\put(40,40){\circle*{4}}
\put(40,60){\circle*{4}}
\put(60,0){\circle*{4}}
\put(60,20){\circle{4}}
\put(60,40){\circle*{4}}
\put(60,60){\circle*{4}}
\put(80,0){\circle*{4}}
\put(80,20){\circle{4}}
\put(80,40){\circle*{4}}
\put(80,60){\circle*{4}}
\put(100,0){\circle*{4}}
\put(100,20){\circle{4}}
\put(100,40){\circle*{4}}
\put(100,60){\circle*{4}}
\put(0,65){$F_2$}
\put(120,5){$F_1$}
\end{picture}
$$
\end{ex}

\begin{ex}
Example of saturated, but not affine facets can be obtained from 
consideration the cone
$$
\sigma^{\vee} = \{ \QQ_{\geq 0}(1,1,1) + 
\QQ_{\geq 0}(1,1,-1) + \QQ_{\geq 0}(1,-1,1) + \QQ_{\geq 
0}(1,-1,-1)\}
$$ in $3$-dimensional space ("infinite" tetrahedral 
pyramid) and monoid~$P$, which is the intersection of 
$\sigma^{\vee}$ and $M = \ZZ^3$. Easy to see, that every facet of 
$\sigma^{\vee}$ is saturated, but the intersection of any 3 facets 
has dimension $0$. It is equal to the point $(0, 0, 0)$.
\end{ex}

\begin{lemma}
Let $\delta$ be a homogeneous LND with slice for which the subalgebra corresponding to the intersection of the given facet $F$ and the lattice $M$ is the kernel. Then $F$ is affine and for the Demazure root $e$ corresponding to LND $\delta$ the vector $-e$ lies in the intersection of all other facets.
\end{lemma}
\begin{proof}
Since there is LND for which this facet $F$ is the kernel, then $F$ is saturated. An equality $\delta(\chi^m) = \langle e, n_{\rho} \rangle \chi^{m+e} = 1$, where $\rho$ is the ray in $\sigma$ corresponding to F, holds if and only if $-e$ lies inside the cone~$\sigma^{\vee}$, 
Therefore we have $\langle -e, n_{\rho_i} \rangle \geq 0$, where $n_{\rho_i}$ 
are primitive integer vectors on extremal rays $\rho_i$ of the cone $\sigma$. On the other hand, by the
definition of the Demazure root, $\langle e, n_{\rho_i} \rangle \geq 0$ for all $\rho_i$
, except the one corresponding to the given facet $F$. Hence we get
that $\langle e, n_{\rho_i} \rangle = 0$ for all $\rho_i$ except the one corresponding to the given facet.
Therefore,~$-e$ must lie in all faces except the given one. Thus their intersection must have dimension $\geq 1$. On the other hand, their intersection have dimension $\leq 1$. Indeed, if the intersection of all other facets have dimension $\geq 2$, then intersection of all facets have a dimension $\geq 1$, which contradict to the assumption that cone $\sigma^{\vee}$ is pointed.
\end{proof}

This statement is equivalent to the fact that all extremal rays of the~$\sigma^{\vee}$, except one, belongs to F.

Consider now an affine facet F of $\sigma^{\vee}$, corresponding to the extremal ray $\rho$ in $\sigma$. Denote by $\tau$ the only extremal ray not belonging to it.

\begin{definition}
Extremal ray $\tau \in \sigma^{\vee}$, corresponding to affine facet $F$ is called {\it affine}, if its primitive integer vector $r$ lies in $P$, satisfies the condition $\langle r, n_{\rho} \rangle = 1$, where $n_{\rho}$ is primitive integer vector on the extremal ray of $\sigma$, corresponding to $F$, and for every hole $h$ in $F$ and $k \in \NN$ points $h + kr $ is also a hole.
\end{definition}

\begin{ex}
Although both facets in Example 1 are affine, the rays 
corresponding to them are not affine, because their primitive vectors are not "closest" integer vectors to these facets - $\langle (1,0), n_{\rho_2} \rangle = 2$, $\langle (1,2), n_{\rho_1} \rangle = 2$, where $n_{\rho_1} = (0, 1)$ and $n_{\rho_2} = (2, -1)$, corresponding to $F_1$ and $F_2$ respectively. On the contrary, in Example 2, a ray, corresponding to facets $F_2$, is affine.
\end{ex}

\begin{ex}
Let us take the monoid $P$ from Example 2 and remove the point $(0, 2)$. Facet $F_2$ is affine, primitive vector $(1,0)$ of the corresponding ray $F_1$ belongs to $P$, satisfies condition  $\langle (1,0), n_{\rho_1} \rangle = 1$, where $n_{\rho_1} = (1,0)$, but for hole $h = (0,2)$ exist infinitely many $k \in \NN$ such that $h + k(1,0)$ is not a hole. Hence, in this case, a ray $F_1$, considered as a ray corresponding to the affine facet $F_2$, is not affine.
$$
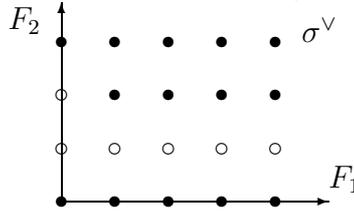
\begin{picture}(120,70)
\put(110,60){$\sigma^{\vee}$}
\put(20,0){\vector(1,0){100}}
\put(20,0){\vector(0,1){75}}
\put(20,0){\circle*{4}}
\put(20,20){\circle{4}}
\put(20,40){\circle{4}}
\put(20,60){\circle*{4}}
\put(40,0){\circle*{4}}
\put(40,20){\circle{4}}
\put(40,40){\circle*{4}}
\put(40,60){\circle*{4}}
\put(60,0){\circle*{4}}
\put(60,20){\circle{4}}
\put(60,40){\circle*{4}}
\put(60,60){\circle*{4}}
\put(80,0){\circle*{4}}
\put(80,20){\circle{4}}
\put(80,40){\circle*{4}}
\put(80,60){\circle*{4}}
\put(100,0){\circle*{4}}
\put(100,20){\circle{4}}
\put(100,40){\circle*{4}}
\put(100,60){\circle*{4}}
\put(0,65){$F_2$}
\put(120,5){$F_1$}
\end{picture}
$$
\end{ex}

\begin{prop}
The facet $F$ admits a homogeneous LND with slice if
and only if it is affine and the corresponding ray $\tau$ is affine.
\end{prop}
\begin{proof}
Suppose that $F$ admits homogeneous LND with slice $\delta$. Then from Lemma 5.2 it follows, that $F$ is affine. Consider the corresponding to $F$ ray $\tau$. In Lemma 5.2 we obtain that for Demazure root $e$, corresponding to $\delta$, is true that $-e$ belongs to the ray $\tau$. It is true that this integer vector  satisfies the condition $\langle -e, n_{\rho} \rangle = 1$ and, therefore, is primitive. It follows from the definition of Demazure root. So we need to proof that $-e \in P$ and that for every hole $h \in F$ and $k \in \NN$ is true that $h + k(-e)$ is also a hole.

Firstly, assume that $-e \not\in P$. From the another hand, on the ray $\tau$ there exist nonzero points, which belong to $P$, opposite contradicts to fact that $P$ is finitely generated. In terms of monoid $P$ derivation $\delta$ shift these points on a vector $e$. Therefore, in a finite number of steps we reach the point $-e$ (since $-e$ is a primitive vector, then all points on the ray $\tau$ that belong to $P$ have the form $k(-e)$ for some $k \in \NN$). Since the derivation $\delta$ is well-defined, then $-e$ must belong to $P$.

For the same reasons for every hole $h \in F$ and $k \in \NN$ is true that $h + k(-e)$ is also a hole (because otherwise in a finite number of shifts from some nonhole point we reach $h$, which is a contradiction to the fact that the derivation $\delta$ is well-defined).

Now let suppose that $F$ is affine and corresponding ray $\tau$ is also affine. Denote by $r$ the primitive integer vector on the ray $\tau$. We claim that the derivation given by the formula 
\begin{equation}
\delta(\mu\chi^m) = \mu \langle m, n_{\rho} \rangle \chi^{m-r} 
\end{equation}
(where $n_{\rho}$ is the primitive integer vector on the extremal ray from cone~$\sigma$, 
which corresponds to $F$) for every homogeneous element $\mu \chi^m$ from $\KK[X]$ is an LND with slice and that it is 
well defined. 

In terms of monoid $P$ derivation $\delta$ is well-defined if and only if shift from any point lying in $P \backslash (P \cap F)$ gives us a point from $P$. Assume that for some $p \in P$ the point $p-r$ is a hole. But for some $k' \in \NN$ the point $p - k'r  = f \in F$. On the one hand, if $f \not \in P$, i.e. is a hole, then all point of the form $f + kr$ for every $k \in \NN$ are holes. It follows from the definition of an affine vector. But this is a contradiction with the fact that $p \in P$. On the other hand, if $f \in P$, then all point of the form $f + kr$ for every $k \in 
\NN$ also belong to $P$, but this is contradiction with the fact that $p-r \not \in P$. Hence, formula (1) gives us a well-defined derivation.

Easy to see that it is LND corresponding to Demazure root $-r$. It is obvious, that $\chi^r$ is slice for this derivation. So we prove that $\delta$ is well-defined.

\end{proof}

In fact, Proposition 5.3 states that the monoid is the product of the monoid corresponding to the intersection $F \cap P$ and all integer points of the ray $\tau$. Points belonging to $P$ generate series of points also belonging to $P \cap F$, and holes belonging to $F$ generate series of holes.

\begin{re}
It follows from Lemma 5.1 and Proposition 5.3 that such a structure of the monoid $P$ is also sufficient condition for the algebra $\KK[X]$ to admit a $\LND$ with slice.
\end{re}

\begin{lemma}\label{lind}
Any nonzero vector from the fixed affine ray $\tau$ cannot be expressed as a linear combination of vectors from other affine rays, i.e all affine rays are linearly independent. 
\end{lemma}
\begin{proof}
All affine rays, except $\tau$, belong to the corresponding affine facet~$F$. Hence all linear combinations of vectors from these rays lie in hyperplane $H$, containing $F$. We have $H \cap \tau = 0$, which gives us the required.
\end{proof}

\begin{corollary}\label{nrays}
There are at most $n$ distinct affine rays.
\end{corollary}
\begin{proof}
Obviously follows from the previous lemma.
\end{proof}

\begin{corollary}\label{mlx}
Suppose there are $k$ affine rays in $n$-dimensinal cone $\sigma^{\vee}$. Then $\ML^*_h(X) = \ML^*(X) \cap \ML_h(X)$ is the subalgebra, corresponding to the monoid $P \cap \tau_{\ML^*_h}$, where $\tau_{\ML^*_h}$ is $(n-k)$-dimensional face of $\sigma^{\vee}$, generated by all non-affine rays.
\end{corollary}
\begin{proof}
The kernels of homogeneous LNDs with slices are subalgebras corresponding to affine facets. Affine ray cannot lie at the intersection of these facets, because their intersection with the corresponding face equal to 0. Hence, the intersection of facets, corresponding to kernels, consists only of non-affine rays. On the other hand, if a facet corresponds to a homogeneous LND with slice, then it contains all rays except the given one affine ray, i.e. including all non-affine ones. Note that the intersection of all affine facets have the dimension~$n-k$, because corresponding affine rays are linearly independent by the Lemma~\ref{lind}.
\end{proof}

\begin{corollary}\label{form}
If an $n$-dimensional affine toric variety $X$ admits LND with slice, then it have the form $X' \times \AA^k$, where $X'$ is the $(n-k)$ dimensional affine toric subvariety, corresponding to the algebra $\ML^*_h(X)$ and $k$ is amount of different affine rays.
\end{corollary}
\begin{proof}
We have $\KK[X] = \ML^{*}_{h}(X) \otimes \KK[\chi_1, \ldots, \chi_k]$, where $\chi_i$ are slices, corresponding to different affine vectors.
\end{proof}

\begin{corollary}
$\ML^*_h(X) = \KK$ if and only if $X = \AA^n$.
\end{corollary}
\begin{proof}
It is obvious, that $\ML^*_h(\AA^n) = \KK$. Inverse follows from Corollary~\ref{form}.
\end{proof}

Corollaries above allow us to prove the criterion of {\it rigidity} of the subvariety $X'$. Recall that variety $X$ called {\it rigid} if $\KK[X]$ does not admit any LND.  

\begin{corollary}
Subvariety $X'$, corresponding to $\ML^*_h(X)$, is rigid affine toric variety if and only if $\ML^*_h(X) = \ML_h(X)$.
\end{corollary}
\begin{proof}
Suppose, firstly, that $\ML^*_h(X) = \ML_h(X)$. Consider the cone~$\tau_{\ML^*_h}$. Assume that the variety $X'$ admit at least one LND. Then it admit at least one homogeneous LND. Therefore, from Lemma~4.2 it follows that there is at least one almost saturated facet $F$ in the cone~$\tau_{\ML^*_h}$. Then by adjoining to $F$ all affine rays we can construct new saturated facet $F_{\sigma^{\vee}}$ in origin cone $\sigma^{\vee}$. Hence, according to Lemma~4.2, there exists such homogeneous LND $\delta$, that $\Ker(\delta)$ is the subalgebra of $\KK[X]$, corresponding to the facet $F_{\sigma^{\vee}}$. Therefore, in terms of the cone $\sigma^{\vee}$, $\ML_h(X)$ corresponds to the intersection $\tau' \cap P$ for a some face $\tau' \preccurlyeq F \prec \tau_{\ML^*_h}$, which contradicts to the fact that $\ML^*_h(X) = \ML_h(X)$.

Suppose now the opposite, i.e that $X'$ is rigid variety. Consider again the cone $\tau_{\ML^*_h}$. Assume that $\ML^*_h(X) \neq \ML_h(X)$. Then there is almost saturated facet $F$ in the cone $\sigma^{\vee}$ such that it don't correspond to any homogeneous LND with slice. But all such facets are obtained by adjoining to some facet of the cone $\tau_{\ML^*_h}$ all affine rays. Therefore, the intersection $F \cap \tau_{\ML^*_h}$ is facet of the cone $\tau_{\ML^*_h}$, which is almost saturated, because contrary $F$ would not be saturated. Hence, $X'$ admits at least one homogeneous LND with slice, which is absurd.
\end{proof}

Now let us consider the structure of $\ML^*(X)$. The main idea will be to construct several LNDs with slice, the intersection of the kernels of which will be equal to $\ML_h(X)$. 

Let $\delta$ be any LND with slice, corresponding to the fixed facet $F$, and $\delta'$ be any homogeneous LND corresponding to any another almost saturated facet $F'$. Let us prove the following lemma.

\begin{lemma}
$\delta + \delta'$ is LND with slice.
\end{lemma}
\begin{proof}
Let us prove, firstly, that $\delta + \delta'$ is LND. Consider in cone $\sigma^{\vee}$ the layers parallel to the facet $F$, i.e. sets $L_i = \{ v \in \sigma^{\vee} | \langle v, n_{f} \rangle = i\}$, where~$f$ is the ray of the cone $\sigma$ corresponding to the facet $F$. It suffices to show that $\delta + \delta'$ is an LND on $L_i$ for any $i$. Note first that for any $v \in L_i$ image $\delta(v) \in L_i$, because in terms of the cone $\sigma^{\vee}$ it is the shift on the vector corresponding to this derivation, which belongs to the facet $F$. $\delta'(v) \in L_{i - 1}$ by the definition of Demazure roots. We are going to prove it by induction on $i$.

Base. Let $i = 0$. $L_0$ is equal to $F$ by definition. By previous remarks, for $v \in L_0$ we have $(\delta + \delta')(v) \in L_0$, and also $(\delta + \delta')(v) = \delta(v)$, because~$F$ is the kernel of $\delta'$. Let $k$ be a natural number such that $\delta^k(v) = 0$. Then $(\delta + \delta')^k(v) = 0$. The base is proven.

Step of induction. Let the goal is proved for all $i < n$. let us prove it for $i = n$. Let $v \in L_i$ and $k$ be a positive integer such that $\delta^k(v) = 0$. Consider $(\delta + \delta')^k(v)$:
$$
(\delta + \delta')^k(v) = \delta^k(v) + \delta^{k-1}(\delta'(v)) + \delta^{k-2}(\delta '(\delta(v))) \ldots + \delta(\delta'^{k-1}(v)) +
$$
$$
+ \delta'^k(v).
$$
Note that all summands except the first one (which is equal to 0) contain $\delta'$ at least once, which reduces the layer by 1. Hence, $(\delta + \delta')(v)$ is a linear combination of elements from the layer $L_{n-1}$ or lower for which the lemma is true by the induction hypothesis. The step is proven.

Consider now a monomial $\chi$ on which $\delta(\chi) = 1$ and $\delta'(\chi) = 0$. It is obvious, that such a monomial exists. Therefore, $\delta + \delta'$ is an LND with slice.
\end{proof}

Note now that any element $f$ from the subalgebra, corresponding to the facet $F$, which belongs to $\Ker(\delta + \delta')$, lies in $F \cap F'$. Indeed, if $f$ has homogeneous components that do not lie in the intersection~$F \cap F'$, then the $\delta$ vanishes them, and the $\delta'$ in terms of the monoid $P$ transfers them in parallel to the corresponding vector $e'$ (Demazure root). It is a contradiction with the fact that $f \in \Ker(\delta + \delta')$. Hence, $\ML^*(X)$ belongs to the subalgebra, corresponding to $F' \cap \tau_{\ML^*_h}$. Thus, we can prove Theorem 1.1(b).

\begin{proof}[Proof of Theorem \ref{m}(b)]
From Theorem \ref{m}(a) we know, that $\ML_{h}(\KK[X]) = \ML(\KK[X])$. Hence, $\ML_{h}(\KK[X]) = \ML(\KK[X]) \subset \ML^{*}(\KK[X])$. On the other hand, we just get that for any almost saturated facet $F'$ Makar-Limanov--Freudenburg invariant $\ML^*(X)$ is contained in the subalgebra, corresponding to $F' \cap \tau_{\ML^*_h}$. Let $F_1, F_2, \ldots, F_m$ be all almost saturated facets. Then $\ML^*(X)$ belongs to the subalgebra, corresponding to 
$$
(F_1 \cap \tau_{\ML^*_h}) \cap (F_2 \cap \tau_{\ML^*_h}) \cap \ldots \cap (F_m \cap \tau_{\ML^*_h}) =
$$
$$
= F_1 \cap F_2 \cap \ldots \cap F_m \cap \tau_{\ML^*_h} = \tau_{\ML_h}.
$$
Therefore $\ML^{*}(\KK[X]) \subset \ML_{h}(\KK[X]) = \ML(\KK[X])$.
\end{proof}

\end{document}